\documentclass[12pt]{amsart}
\usepackage{lineno,hyperref}
\usepackage{amsmath,amsfonts,amsthm}
\usepackage{amssymb,latexsym}
\usepackage{cases}
\usepackage[numbers]{natbib} 
\usepackage{booktabs}

\usepackage{cleveref}
\renewcommand\appendix{\setcounter{secnumdepth}{-2}}

\usepackage{url,cancel}
\usepackage{float}

\usepackage{caption,enumitem}
\newtheorem{thm}{Theorem}[section]

\newtheorem{lem}[thm]{Lemma}
\newtheorem{prop}[thm]{Proposition}

\theoremstyle{remark}
\newtheorem{re}[thm]{Remark}
\newtheorem{ex}[thm]{Example}
\newcommand{\Q}{{\mathbb Q}}

\newcommand{\p}{{\mathfrak p}}

\makeatletter
\@addtoreset{equation}{section}
\makeatother
\numberwithin{equation}{section}

\begin{document}

\author{Zilong He}
\address{Department of Mathematics, University of Hong Kong, Pokfulam, Hong Kong}
\email{zilong.he@hotmail.com}
%\author{Ben Kane}
%\address{Department of Mathematics, University of Hong Kong, Pokfulam, Hong Kong}
%\email{bkane@hku.hk}
\title{Inequalities for inert primes and their applications}
%\thanks{The research of the second author is supported by grant project numbers 17316416, 17301317, and 17303618 of the Research Grants Council of Hong Kong SAR.  Part of the research was also conducted while the second author was supported by grant project number 17302515 of the Research Grants Council of Hong Kong SAR}
\date{\today}
\begin{abstract}
For any given non-square integer $ D\equiv 0,1 \pmod{4} $, we prove Euclid's type inequalities for the sequence $ \{q_{i}\} $ of all primes satisfying the Kronecker symbol $ (D/q_{i})=-1 $, $ i=1,2,\cdots, $ and give a new criterion on a ternary quadratic form to be irregular as an application, which simplifies Dickson and Jones's argument in the classification of regular ternary quadratic forms to some extent.
\end{abstract}
\maketitle

\section{Introduction}	

The proof on infinitude of primes of Euclid indicates that the following inequality
\[ p_{n+1}<p_{1}p_{2}\cdots p_{n} \]
holds for $ n\ge 2 $, where $ p_{i} $ denotes the $ i $-th prime. This can be regarded as the beginning of the study of inequalities of prime numbers. A natural motivation is to seek similar inequalities for some forms of primes. van Golstein Brouwers et al. \cite{brouwers_totally_nodate} developed an approach to study totally Goldbach numbers which is dependent on the conjecture (for which there is numerical evidence) that for every prime $ q $ and $ a=1,\ldots, q-1 $, the inequality \[ r_{n+1}<r_{1}r_{2}\cdots r_{n} \] holds for $ n\ge 3 $, where $ r_{i} $ is $ i $-th prime with $ r_{i}\equiv a \pmod{q}$. However, by a series arguments \cite[p.\hskip 0.1cm 254]{brouwers_totally_nodate}, it was shown that the confirmation for these inequalities is related to the generalization of Schinzel and Sierpinski's conjecture that $ r_{1}<q^{2} $ for all primes $ q $ and all $ a=1,\ldots, q-1 $, which is unsolved so far. By relaxing the restriction on the primes $r_i$ to allow all primes whose Kronecker symbols are negative instead of taking $r_i$ in a fixed congruence class, in this paper we deduce an analogous inequality unconditionally. 

To be precise, let $ D\equiv 0,1\pmod{4} $ be a non-square discriminant and set $d:=|D|$ for convenience. Also, for a given discriminant $ D $, we define the set
\begin{align*}
\mathbb{P}(D)&:=\{q\,:\, \mbox{$ q $ is prime and } (D/q)=-1\},
\end{align*}
where $ (D/\cdot) $ is the Kronecker symbol, and let $ \{q_{i}\}_{D} $ be the sequence of all primes in $\mathbb{P}(D) $ in ascending order. For $ i,j\in\mathbb{Z} $, we also set $  \delta_{i,j}=1 $ if $ i=j $, and $ 0 $ otherwise. Then we prove the following result.

\begin{thm}\label{thm1}
	For a given non-square discriminant $D$, set
	\begin{equation*}
	H(D):=
	\begin{cases}
	d & D\not=-3,-4,5,\\
	11 &  otherwise.
	\end{cases}
	\end{equation*}
	Let $ q_{i_{0}+1} $ be the least prime greater than $ H(D) $ in $ \{q_{i}\}_{D} $. Then the inequality
	\begin{equation}\label{eq10}
	\begin{aligned}
	q_{i+1}<q_{1}q_{2}\cdots q_{i}
	\end{aligned}
	\end{equation}
	holds for $ i\ge i_{0}$. 
\end{thm}

For example, when $ D=-4 $, $ \mathbb{P}(D)=\{3,\,7,\,11,\,19,\,23,\,31,\cdots\} $ and $ q_{1}=3<q_{2}=7<H(D)=q_{3}=11<q_{4}=19 $. Then $ q_{i+1}<q_{1}q_{2}\cdots q_{i} $ holds for $ i\ge 3 $ by Theorem \ref{thm1}. Similarly, when $ D=-12 $, $ q_{1}=5<q_{2}=11<H(D)=12<q_{3}=17 $. It follows that $q_{i+1}<q_{1}q_{2}\cdots q_{i} $ also holds for $ i\ge 2 $ from Theorem \ref{thm1}. Note that $ \{q_{i}\}_{D} $ exactly consists of all the primes of the form  $4\ell-1$ (resp. $6\ell-1$), when $ D=-4 $ (resp.~$ D=-12$). As we know, Molsen \cite{molsen_zur_1941} showed that for $ n\ge 118 $, there exists at least one prime of the form $ 4\ell-1 $ between $ n $ and $ 4n/3 $. Erd\H{o}s \cite{erdos_uber_1935} also proved that for $ n\ge 6 $, there exists at least one prime of the form $ 6\ell -1 $ between $ n $ and $ 2n $. Using these bounds, one can deduce \eqref{eq10} by induction for $ D=-4,-12 $, but, except for asymptotic results, there seem to be no known Bertrand's postulate for general $ D $ in the literature.

A weak version of \eqref{eq10} was first considered by L. E. Dickson \cite{dickson_ternary_1926} in the classification of regular (ternary) quadratic forms $ x^{2}+by^{2}+cz^{2} $, i.e., those quadratic forms which obey the Hasse-Minkowski local-global principle \cite{conrad_localglobal}. Given a positive integer $ b $, let $ p_{i}$'s be the odd primes for which the Diophantine equation $ x^{2}+by^{2}=p_{i} $ has no solution in $ \mathbb{Z} $ and $p_{1}<\cdots<p_{i_{0}}<b<p_{i_{0}+1}< \cdots.$ Dickson proved that the inequality $p_{i+1}<p_{1}p_{2}\cdots p_{i}$ holds for $ i\ge i_{0} $ \cite[footnote, p.\hskip 0.1cm 336]{dickson_ternary_1926} and further excluded irregular quadratic forms by this inequality \cite[Theorem 5]{dickson_ternary_1926}, which is another motivation for Theorem \ref{thm1}. The advantage of Inequaltiy \eqref{eq10} is that the primes in \eqref{eq10} can be described with the Kronecker symbol instead of an integral equation, so that we are able to apply Dickson's argument uniformly, thereby obtaining Theorem \ref{thm5} below. Furthermore, modifying Dickson's approach, the author and B. Kane show that when $ m $ is sufficiently large, there are no primitive regular ternary $ m $-gonal forms by virtue of analogous inequalities involving primes with additional restrictions (see \cite{he_regular_2019} for details). 

Before stating our criterion, we introduce some basic terminology from the theory of quadratic forms.  Let $ F $ be a field and $ R\subset F $ a ring. For an $n$-ary quadratic polynomial $ f(x_{1},\cdots,x_{n})\in F[x_{1},\cdots, x_{n}] $ and $\ell\in F$, we say that $\ell$ is \begin{it}represented by $f$ in $ R $\end{it} if the equation $ f(x)=\ell$ is solvable with $x\in R^n$. Also, when $ F=\mathbb{Q} $, we say that $\ell $ is locally (resp. globally) represented by $ f $ in $ \mathbb{Z}_{p} $ for each prime $ p $ including $ p=\infty $ (resp. in $ \mathbb{Z} $). A quadratic polynomial $ f $ over $ \mathbb{Q} $ is said to be \begin{it}regular\end{it} if it globally represents all rational numbers that are locally represented by $ f $. We also call $ f $ {\it irregular} if $ f $ is not regular.

\begin{thm}\label{thm5}
	Let $ a,b,c $ be positive integers with $ a\le b\le c $ and whose odd parts are pairwise co-prime. Assume that
	\begin{equation}\label{eq12}
	\begin{aligned}
	ax^{2}+by^{2}+cz^{2}\equiv n \pmod{8}
	\end{aligned}  
	\end{equation} 
	is solvable for $ n=1,3,5,7 $. If the form $ ax^{2}+by^{2}+cz^{2} $ is regular, then $ c\le 4ab+3\delta_{ab,1} $.
\end{thm}

\begin{re}
	Under the assumption of Theorem \ref{thm5}, if $ c>4ab+3\delta_{ab,1} $, then $ ax^{2}+by^{2}+cz^{2} $ is irregular. Recall Jones's argument \cite[p.\hskip 0.1cm19--20]{jones_representation_1928} in which the irregular forms with $ a>1 $ can be ruled out by taking $ n=1 $ or $ 2 $ and comparing with the forms in Table I \cite[p.\hskip 0.1cm 125]{jones_representation_1928}. Note that the coefficient $ a $ of each form in Table I is $ 1 $. Hence it is enough to consider $ a=1 $, which is exactly the case Dickson considers in \cite{dickson_ternary_1926}. Those theorems in \cite{dickson_ternary_1926} satisfying the condition \eqref{eq12} will be simplified by Theorem \ref{thm5}.
\end{re}

\section{Proof}\label{sec2}
In Section \ref{sub21}, we will show Theorem \ref{thm1} by using elementary arguments similar to the Euclidean proof of the infinitude of primes, and then we will give a proof of Theorem \ref{thm5} as an application of Theorem \ref{thm1}.

\subsection{Inequalities involving $ \{q_{i}\}_{D} $}\label{sub21}
In this subsection, we always write $ Q_{i}:=q_{1}q_{2}\cdots q_{i} $ for the product of the first $ i $ terms of $ \{q_{i}\}_{D} $ and $Q_{i}:=1$ if $i\leq 0$ for brevity. Also, we denote by $\{p_{i}\}$ the original prime sequence ($p_{1}=2$, $p_{2}=3$, $\cdots$).

As in the introduction, we let $D$ be a non-square discriminant. Since $D$ is not a perfect square, the Kronecker symbol $ (D/\cdot) $ is a nonprincipal character modulo $ d $. Thus there exists some integer $N<d $ such that $ (D/N)=-1 $ and consequently there also exists some prime $ q$ dividing $ N $ such that $ (D/q)=-1 $. Therefore, $\mathbb{P}(D)\neq\emptyset$. Also, $ d\ge 3 $ and $ 2\le q_{1}<d $. Moreover, if $ (D/2)\neq -1 $, then we further have $ d\ge 4 $ and $ 3\le q_{1}<d $. Given $ D $, we also let $\Omega_{D}^{-}(n)$ count the number of prime divisors of $n$ in $ \mathbb{P}(D) $, counting multiplicities. %We also let $\Omega(n)$ count the number of prime divisors of $n$, counting multiplicities.

\begin{lem}\label{lem:newprime}
	Let $D$ be a non-square discriminant. Let $s_1$ and $s_2$ be positive integers and $\gcd(ds_{1},s_{2})=1$. Suppose that $ D<0 $ or $ \Omega_{D}^{-}(s_{2}) $ is odd and set
	\begin{equation*}
	M:=\begin{cases}
	s_{2}-ds_{1}&\text{if } 2\nmid \Omega_{D}^{-}(s_2)\text{ and }s_{2}>ds_{1},\\
	ds_{1}-s_{2}&\text{if }D<0, \hskip 0.2cm 2\mid \Omega_{D}^{-}(s_2),\text{ and } s_{2}<ds_{1},\\
	&\text{or if } D>0, \hskip 0.2cm 2\nmid \Omega_{D}^{-}(s_2),\text{ and }s_{2}<ds_{1}.
	\end{cases}
	\end{equation*}
	Then there exists some prime $ q\in\mathbb{P}(D) $ dividing $ M $ but prime to $ s_{1}s_{2} $ and hence $ q\le M $. 
	
	Moreover, if $ 2\nmid \Omega_{D}^{-}(s_{2}) $, then there exists some prime $ q^{\prime}\in\mathbb{P}(D) $ for which $ q^{\prime}\mid M^{\prime}:=ds_{1}+s_{2} $ but $ \gcd(q^{\prime},s_{1}s_{2})=1 $ and also hence $ q^{\prime}\le M^{\prime}$.
\end{lem}
\begin{proof}
	By construction, we have $M>0 $ and clearly $\gcd(s_{1},s_{2})=1$ from the assumption that $ \gcd(ds_{1},s_{2})=1 $. One can check that
	\begin{align*}
	\gcd(M,s_{1}s_{2})=\gcd(M,s_{1})\gcd(M,s_{2})=\gcd(s_{2},s_{1})\gcd(ds_{1},s_{2})=1.
	\end{align*}
	We then use the periodicity and multiplicativity of the Kronecker symbol to compute, for $M=\pm \left(s_2-ds_1\right)$,
	\[
	\left(\dfrac{D}{M}\right) = \left(\dfrac{D}{\pm s_2}\right)=\left(\dfrac{D}{\pm 1}\right)\left(\dfrac{D}{s_{2}}\right)=\left(\dfrac{D}{\pm 1}\right)(-1)^{\Omega_{D}^{-}(s_2)}.
	\]
	In each case $(D/M)=-1$ by the definition of $M$, so there is a prime $q$ dividing $M$ for which $(D/q)=-1$, and $q\nmid s_1s_2$ since $\gcd(M,s_1s_2)=1$. 
	
	Clearly, $ M^{\prime}>0 $. Similarly, one can check that $ \gcd(M^{\prime},s_{1}s_{2})=\gcd(s_{2},s_{1})\gcd(ds_{1},s_{2})=1$ and $ (D/M^{\prime})=(D/s_{2})=-1 $ because $ 2\nmid \Omega_{D}(s_{2}) $. Hence there exists a prime $ q^{\prime} $ dividing $ M^{\prime} $ and $ \gcd(q^{\prime},s_{1}s_{2})=1 $. 
\end{proof}  

For clarifying Lemma \ref{lem:newprime}, we illustrate some examples.

\begin{ex}\label{ex:newprime}
	
	Consider the sets $ \mathbb{P}(D) $ for $ D=-3,5 $.
	
	 (1) When $ D=-3<0 $, $ d=3 $ and $ \mathbb{P}(D)=\{2,\,5,\,11,\,17,\,23,\,29,\,41,\cdots\}$,
	 
	 (a) for $ M=29-2d $ with $ (s_{1},s_{2})=(2,29) $, $ M>0 $, $ \Omega_{D}^{-}(s_{2})=1 $, $ 23\mid M $ and $ 23\nmid s_{1}s_{2} $;
	 
	 (b) for $ M=17d-22 $ with $ (s_{1},s_{2})=(17,2\cdot 11) $, $ M>0 $, $ \Omega_{D}^{-}(s_{2})=2 $, $ 29\mid M $ and $ 29\nmid s_{1}s_{2} $;
	 
	 (c) for $ M^{\prime}=3d+41 $ with $ (s_{1},s_{2})=(3,41) $, $ \Omega_{D}^{-}(s_{2})=1 $, $ 5\mid M^{\prime} $ and $ 5\nmid s_{1}s_{2} $.
	
	 \noindent (2) When $ D=5>0 $, $ d=5 $ and $   \mathbb{P}(D)=\{2,\,3,\,7,\,13,\,17,\,23,\,37,\,43,\cdots\}$,
	 
	  (a) for $ M=78-7d $ with $ (s_{1},s_{2})=(7,2\cdot 3\cdot 13) $, $ M>0 $, $\Omega_{D}^{-}(s_{2})=3$, $ 43\mid M $ and $ 43\nmid s_{1}s_{2} $;
	  
	  (b) for $ M=8d-17 $ with $(s_{1},s_{2})=(8,17)$, $ M>0 $, $ \Omega_{D}^{-}(s_{2})=1 $, $ 23\mid M $ and $ 23\nmid s_{1}s_{2} $;
	  
	  (c) for $ M^{\prime}=3d+22 $ with $ (s_{1},s_{2})=(3,2\cdot 11) $, $ \Omega_{D}^{-}(s_{2})=1 $ (as $ (D/11)=1 $), $ 37\mid M^{\prime} $ and $ 37\nmid s_{1}s_{2} $.
\end{ex}

As seen above, we are able to construct a positive integer divisible by a prime $ q $ for which $ (D/q)=-1 $, $ q\nmid s_{1}s_{2} $ and $ q $ is bounded by $ds_{1} $ or $ s_{2} $ or $ ds_{1}+s_{2} $ by Lemma \ref{lem:newprime}. Hence given a term in $ \{q_{i}\}_{D} $ (except for $ q_{1} $ and $ q_{2} $), it is possibly bounded by the product of the previous terms, as long as the values of $ s_{1} $ and $ s_{2} $ are appropriately chosen from some of the previous terms so that $ ds_{1} $ or $ s_{2} $ or $ ds_{1}+s_{2} $ is bounded by the product of the chosen primes. For instance, given $ \{q_{i}\}_{D} $ with $ D=5 $, we see that $ q_{8}=43\mid 78-7d<2\cdot 3\cdot 7\cdot 13<\prod_{i=1}^{7}q_{i} $ from Example \ref{ex:newprime} (2)(a). 

To apply the first assertion in Lemma \ref{lem:newprime}, we require the condition $ M>0 $. However, it is not easy to determine $ ds_{1}>s_{2} $ or $ ds_{1}<s_{2} $ in general. Therefore we also need the following lemma.

 \begin{lem}\label{lem:turning_index}
   Let $ \{q_{i}\}_{D} $ be the prime sequence associated with a given non-square discriminant $ D $. Let $M$ be a positive integer and $ M\ge d $. Then there exists a unique integer $ n\ge 1 $ depending on $ D $ and $ M $ such that
	\[	q_{1}q_{2}\cdots q_{n}<M
	\qquad\text{and}\qquad
	q_{1}q_{2}\cdots q_{n+1}>M.\]
\end{lem}
\begin{proof}
	For given the sequence $ \{q_{i}\}_{D} $ (associated with $ D $), define the set $ A $ by all the products $ Q_{i}=q_{1}q_{2}\cdots q_{i} $ less than $ M $; that is
	\[ A:=\{Q_{i}:Q_{i}<M,\,i=1,2,\cdots\}. \]
	Then $q_{1}\in A$ because $q_1<d\le M$. Hence $ \emptyset\not=A\subseteq\mathbb{N} $ and it is bounded from above. By the well-ordering principle, there exists a unique maximal element in $ A $, say $ Q_{n}$. It follows that $Q_{n}<M $ and $ Q_{n+1}>M$ from the maximality of $ Q_{n} $.
\end{proof}

 We call the unique $ n $ satisfying  Lemma \ref{lem:turning_index} the \begin{it}turning index\end{it} of $ D $ and $ M $, and denote it by $ i(D,M) $, or simply $ i(D) $ when $ M=|D| $. To obtain a bound of $ q_{i(D)+1} $, we need an inequality involving primes given by Panaitopol \cite[Corollary]{panaitopol_inequality_2000}, which is a generalization of Bonse's inequality \cite[p.\hskip 0.1cm 187--192]{rademacher_enjoyment_1957}.
\begin{prop}[Panaitopol]\label{prop20}
Let $ \ell\ge 1 $ be an integer. Then $ p_{n+1}^{\ell}<p_{1}p_{2}\cdots p_{n} $ holds for $ n\ge 2\ell $.
\end{prop}

\begin{lem}\label{lem22}
		Let $ \{q_{i}\}_{D} $ be the prime sequence associated with a given non-square discriminant $ D $. Then $ q_{i(D)+1}<d$ except for $D=-3,-4$. More precisely,
\noindent

\noindent
\begin{enumerate}[leftmargin=*,label={\rm(\arabic*)}]	
\item if $ D<0 $ and $ 2\mid i(D) $ or $ D>0 $ and $ 2\nmid i(D) $, then $ q_{i(D)+1}\le d-Q_{i(D)} $;
\item if $ D>0 $ and $ 2\mid i(D) $, then $ q_{i(D)+1}\le d-Q_{i(D)-2} $;
\item if $ D<0$ and $ 2\nmid i(D) $, then $ q_{i(D)+1}\le d-Q_{i(D)-2} $, except for $ D=-3,-4 $.
 \end{enumerate}
\end{lem}

\begin{proof}
	 Write $ n_{0}=i(D) $. First, we have $ Q_{n_{0}}<d $ and $ Q_{n_{0}+1}>d $ from Lemma \ref{lem:turning_index}.
	
\noindent
	{\rm (1)} If $ D<0 $ and $ 2\mid n_{0} $ or $ D>0 $ and $ 2\nmid n_{0} $, then take $ s_{1}=1 $ and $ s_{2}=Q_{n_{0}}=q_{1}q_{2}\cdots q_{n_{0}} $. By Lemma \ref{lem:newprime}, we have
	$q_{n_{0}+1}\le q_{j}\mid d-Q_{n_{0}}$ for some $ j\ge n_{0}+1 $.
	
\noindent
	{\rm (2)}  Suppose $ q_{n_{0}+1}>d-Q_{n_{0}-2} $. As $ D>0 $ and $ n_{0}-1 $ is odd, Lemma \ref{lem:newprime} implies that there exists $q_{j}\mid d-Q_{n_0-1}$ with $j\geq n_0$. Since $d-Q_{n_0-1}<d-Q_{n_{0}-2}$ and $q_{n_0+1}>d-Q_{n_{0}-2}$, we see that $j=n_0$ and hence $ q_{n_{0}}\mid d-Q_{n_{0}-1} $.
	
	Assume that $ n_{0}\ge 4 $. Since $ 2 \nmid \Omega_{D}^{-}(q_{i}Q_{n_{0}-2}) $ ($ i=1,2 $), Lemma \ref{lem:newprime} implies that 
	\begin{equation*}
	\begin{cases} 
	q_{j_{1}}\mid d-q_{1}Q_{n_{0}-2}, \\
	q_{j_{2}} \mid d-q_{2}Q_{n_{0}-2}.
	\end{cases}	
	\end{equation*}
	for some $ n_{0}-1\le j_{1},j_{2}\le n_{0} $, again using the fact that $ q_{n_{0}+1}>d-Q_{n_{0}-2}$ and $q_{j_i}\le d-q_{1}Q_{n_{0}-2}$.  If $ j_{1}=j_{2}=n_{0}-1 $, then 
	\[
	q_{n_{0}-1}\mid (d-q_{1}Q_{n_{0}-2})-(d-q_{2}Q_{n_{0}-2})=Q_{n_{0}-2}\!\left(q_{2}-q_{1}\right).
	\]
	It follows that $q_{n_{0}-1}\mid q_{2}-q_{1} $, a contradiction. Without loss of generality, we thus have $ j_{1}= n_{0} $. Then $q_{n_{0}}\mid Q_{n_{0}-2}(q_{n_{0}-1}-q_{1})$ follows from $ q_{n_{0}}\mid d-Q_{n_{0}-1}$. This implies $ q_{n_{0}}\mid q_{n_{0}-1}-q_{1} $, which is also impossible. Hence $ n_{0}=2$.
    
    Now $ q_{3}>d-1 $, $ q_{1}\mid d-q_{2} $ and $ q_{2}\mid d-q_{1} $. If $ q_{1}^{2}q_{2}<d $, then $ q_{3}\le q_{j}\mid d-q_{1}^{2}q_{2}<d-1  $ for some $j\geq 3$ by Lemma \ref{lem:newprime}, which contradicts $ q_{3}>d-1 $. So $ d<q_{1}^{2}q_{2} $. Now suppose $ q_{1}^{3}<d $. Then $ q_{2}\mid d-q_{1}^{3} $ by Lemma  \ref{lem:newprime} (as $ q_{3}>d $). It follows that
    \begin{align*}
    	q_{2}\mid (d-q_{1})-(d-q_{1}^{3})=q_{1}(q_{1}-1)(q_{1}+1)
    \end{align*}
    and hence $ q_{2}\mid q_{1}+1 $. We must therefore have $ q_{1}=2 $ and $ q_{2}=3 $. Thus $ 8=q_{1}^{3}<d<q_{1}^{2}q_{2}=12 $. However there are no positive non-square discriminants in this range. We deduce that $ d<q_{1}^{3} $.
	For any prime $p<q_{1} $, note that $ pq_{2}<Q_{2}<d $ (as $ n_{0}=2 $). If $p\nmid D$, then we must have $(D/p)=1$, and hence since $ \Omega_{D}^{-}(pq_{2})=1 $ and $ D>0 $ Lemma \ref{lem:newprime} implies that $ q_{1}\mid d-pq_{2} $. Thus  
	\begin{align*}
		q_{1}\mid (d-q_{2})-(d-pq_{2})=(p-1)q_{2},
	\end{align*}
	which implies $ q_{1}\mid p-1 $. This is impossible because $q_1>p$. Thus we conclude that $p\mid d $ for all primes $p<q_1$. This implies that 
	\begin{align*}
		\prod\limits_{p<q_{1}}p\mid d<q_{1}^{3}.
	\end{align*}
	But when $ q_{1}\ge 17=p_{7} $, the inequality 
	\begin{align*}
			q_{1}^{3}<\prod\limits_{p<q_{1}}p
	\end{align*}
	holds by Proposition \ref{prop20} ($ \ell=3 $). Hence we only need to consider $ q_{1}\in \{2,3,5,7,11,13\} $. In fact, the value of $ D $ can be determined by the relation $ q_{1}^{2}<Q_{2}<d<q_{1}^{3} $. Hence one can check that only $ D\in \{5,8,12\} $ satisfies $ q_{3}>d-q_{1} $, but the turning index is $ i(D)=1 $ for these $ D $, yielding (2).

\noindent
	{\rm (3)} Suppose that $ q_{n_{0}+1}>d-Q_{n_{0}-2} $. Since $ D<0 $ and $ n_{0}-1 $ is even, Lemma \ref{lem:newprime} implies that there exists $q_{j}\mid d-Q_{n_0-1}$ with $j\geq n_0$. Since $d-Q_{n_0-1}\le d-Q_{n_{0}-2}$ and $q_{n_0+1}>d-Q_{n_{0}-2}$, we see that $j=n_0$ and hence $ q_{n_{0}}\mid d-Q_{n_{0}-1} $.

 Assume that $ n_{0}\ge 5 $. Since $ 2 \mid \Omega_{D}^{-}(q_{i}Q_{n_{0}-2}) $ ($ i=1,2 $), Lemma \ref{lem:newprime} implies that 
	\begin{equation*}
	\begin{cases}

	q_{j_{1}}\mid d-q_{1}Q_{n_{0}-2}, \\
	q_{j_{2}} \mid d-q_{2}Q_{n_{0}-2},
	\end{cases}	
	\end{equation*}
for some $ n_{0}-1\le j_{1},j_{2}\le n_{0} $, again using the fact that $ q_{n_{0}+1}>d-Q_{n_{0}-2}$ and $q_{j_i}<d-q_{1}Q_{n_{0}-2}$.  If $ j_{1}=j_{2}=n_{0}-1 $, then 
\[
	q_{n_{0}-1}\mid (d-q_{1}Q_{n_{0}-2})-(d-q_{2}Q_{n_{0}-2})=Q_{n_{0}-2}\!\left(q_{2}-q_{1}\right).
\]
	It follows that $q_{n_{0}-1}\mid q_{2}-q_{1} $, a contradiction. If $ j_{1}= n_{0} $ (resp. $j_2=n_0$), then $q_{n_{0}}\mid Q_{n_{0}-2}(q_{n_{0}-1}-q_{1})$ (resp. $q_{n_0}\mid Q_{n_0-2}(q_{n_0-1}-q_2)$) follows from $ q_{n_{0}}\mid d-Q_{n_{0}-1}$. This implies that $ q_{n_{0}}\mid q_{n_{0}-1}-q_{1}$ (resp. $q_{n_0}\mid q_{n_0-1}-q_2$), which is also impossible because of the assumption that $n_0>3$. Hence $ n_{0}=1 $ or $ n_{0}=3 $.
	
When $n_0=1$, we have $Q_{n_0-2}=1=Q_{n_0-1}$ and we have assumed that  $ q_{2}>d-1$ and shown that $ q_{1}\mid d-1 $. Then $ q_{1}\mid q_{2}-d $ by Lemma \ref{lem:newprime}. It follows that $ q_{1}+d\le q_{2} $. Again by Lemma \ref{lem:newprime}, $ q_{j} \mid d+q_{1} $ for some $ j\ge 2 $ and so $ q_{2}\le d+q_{1} $. Hence $ q_{2}=d+q_{1} $. Suppose $ q_{1}^{2}<d $. Then Lemma \ref{lem:newprime} implies that for some $j\ge 2$ we have 
\[
 d+q_{1}=q_{2}\le q_{j}\mid d-q_{1}^{2},
\]
a contradiction. Hence $ q_{1}<d<q_{1}^{2} $. Assume $ q_{1}\ge 11 $. Note that for any prime $ p<q_{1} $, if $p\nmid D$, then $ q_{1}\mid d-p $ by Lemma \ref{lem:newprime}, as $ \Omega_{D}^{-}(p) $ is even. It follows that $ q_{1}\mid p-1 $ from $ q_{1}\mid d-1 $, which is impossible. This implies that $ p\mid d $ for any prime $ p<q_{1} $. Since $ q_{1}\ge 11=p_{5} $,
	\begin{align*}
		q_{1}^{2}<\prod\limits_{p<q_{1}}p=\prod\limits_{\substack{p\mid d\\p<q_{1}}}p\mid d
	\end{align*}
	by Proposition \ref{prop20} (with $ \ell=2 $). This contradicts the fact that $ d<q_{1}^{2} $. For each $ q_{1}\in \{2,3,5,7\} $, there are only finitely many non-square discriminants $ D<0$ satisfying $ q_{1}<d<q_{1}^{2} $. By directly checking, one can see that only $ D=-3,-4 $ satisfy $ i(D)=1 $ and $ q_{2}>d-1 $.
	
	When $ n_{0}=3 $, we have $ Q_{n_{0}-2}=q_{1} $, $ Q_{n_{0}-1}=q_{1}q_{2} $, $ q_{3}\mid d-q_{1}q_{2} $ and the assumption $ q_{4}>d-q_{1} $. Suppose that $ d>q_{1}^{3}q_{2} $. Then $ q_{3}\mid d-q_{1}^{3}q_{2} $ by Lemma \ref{lem:newprime}, since $ q_{4}>d-q_{1}^{3}q_{2} $. Hence
\[
q_{3}\mid (d-q_{1}q_{2})-(d-q_{1}^{3}q_{2})=q_{1}q_{2}(q_{1}-1)(q_{1}+1), 
\]
a contradiction. So $ d<q_{1}^{3}q_{2} $ and hence $ q_{1}q_{2}q_{3}=Q_{n_{0}}<d<q_{1}^{3}q_{2} $ (with $n_{0}=3$). It follows that $ q_{2}<q_{3}<q_{1}^{2} $ and we further have $ q_{1}^{3}<Q_{3}<d<q_{1}^{3}q_{2}<q_{1}^{5} $. Since $q_2q_3<Q_3<d$ (using the definition of $n_0=3$), Lemma \ref{lem:newprime} and $ q_{4}>d-q_{1} $ imply that $ q_{1}\mid d-q_{2}q_{3}>0 $. Assume that $ q_{1}\ge 31$. For any prime $ p<q_{1} $, if $ p\nmid D $, then $ (D/p)=1 $. Since $ D<0 $, $ d>Q_{3}>pq_{2}q_{3} $ and $ \Omega_{D}^{-}(pq_{2}q_{3}) $ is even, we have $ q_{1}\mid d-pq_{2}q_{3} $ by Lemma \ref{lem:newprime}. This implies
	\[q_{1}\mid (d-q_{2}q_{3})-(d-pq_{2}q_{3})=(p-1)q_{2}q_{3}\]
	and so $ q_{1}\mid p-1 $, which contradicts $ p<q_{1} $. Hence $ p\mid d $ for any prime $ p<q_{1} $. Since $ q_{1}\ge p_{11}=31 $, we conclude from Proposition \ref{prop20} (taking $ \ell=5 $) that
\[
q_{1}^{5}<\prod_{p<q_{1}}p=\prod\limits_{\substack{p\mid d\\p<q_{1}}}p\leq d.
\]
This contradicts $ d<q_{1}^{5} $ and hence
	\[ q_{1}\in \{2,3,5,7,11,13,17,19,23,29\}.\]
	 For $3\le d\le 23^{5} $, one can check that only \[ D\in \{-3,-4,-7,-8,-11,-12,-15,-16,-19,-20,-24\} \]
	 satisfies $ D\equiv 0,1\pmod{4} $ and $ q_{4}>d-q_{1} $, but the turning index $ i(D)\not=3 $ for these $ D $. For $ 23^{5}<d<29^{5} $, a direct computer check shows that there does not exist a choice of $ D $ simultaneously satisfying the conditions $ D\equiv 0,1\pmod{4} $, $ q_{1}=29 $ and $ q_{4}>d-q_{1} $.
\end{proof}

We now begin bounding the primes $q_i$ in the sequence of primes from $\mathbb{P}(D)$ in terms of the products of previous primes from the sequence. 
\begin{lem}\label{lem23}
	 Let $ \{q_{i}\}_{D} $ be the prime sequence associated with a given non-square discriminant $ D $. Then $ q_{i+1}<q_{1}q_{2}\cdots q_{i} $ holds for $ i\ge i(D)+1 $ except for $ D=-3,5 $. In particular, we have the following:
\noindent

\noindent
\begin{enumerate}[leftmargin=*,label={\rm(\arabic*)}]		
\item if $ 2\mid i(D) $, then $ d+q_{i+1}\le q_{1}q_{2}\cdots q_{i} $ holds for $ i\ge i(D)+1 $; 
\item if $ 2\nmid i(D) $, then $ d+q_{i+1}\le q_{1}q_{2}\cdots q_{i} $ holds for $ i\ge i(D)+2 $;
\item if $ D>0 $ and $ 2\nmid i(D) $, then $ q_{i(D)+2}<q_{1}q_{2}\cdots q_{i(D)+1} $ holds except for $ D=5 $;
\item if $ D<0 $ and $ 2\nmid i(D) $, then $ q_{i(D)+2}<q_{1}q_{2}\cdots q_{i(D)+1} $ holds except for $ D=-3 $.
\end{enumerate}
\end{lem}

\begin{proof}
	Write $ n_{0}=i(D) $. Note that $ Q_{n_{0}}<d $ and $ Q_{n_{0}+1}>d $ by Lemma \ref{lem:turning_index}.

\noindent
(1) If $ 2\mid n_{0} $, then for $ i\ge n_{0}+1 $, consider
	 \begin{equation}\label{eqn20}
	 {N_{i}=}
	 \begin{cases}
	 Q_{i}-d &  2\nmid i,  \\
	 dq_{1}+Q_{i}/q_{1} &  2\mid i.
	 \end{cases}	
	 \end{equation}
	 Clearly, $ N_{i}>0$, as $ Q_{i}\ge Q_{n_{0}+1}>d $. If $ 2\nmid i $, then by Lemma  \ref{lem:newprime} we have $q_{i+1}\le Q_{i}-d$, i.e., $ d+q_{i+1}<q_{1}\cdots q_{i} $. If $ 2\mid i $, then take $ s_{1}=q_{1}$ and $ s_{2}=q_{2}\cdots q_{i}=Q_{i}/q_{1}$ in Lemma \ref{lem:newprime}, it follows that 
	\begin{align*}
	q_{i+1}\le q_{j_{i}}\mid N_{i}=dq_{1}+q_{2}\cdots q_{i} 
	\end{align*}
	for some $ j_{i}\ge i+1 $. As $ 2\mid n_{0} $, $ i \ge n_{0}+2 $. It follows that $ Q_{i-1}\ge Q_{n_{0}+1}>d $.
	 Hence
	\begin{align*}
	d+q_{i+1}\le d(q_{1}+1)+\dfrac{Q_{i}}{q_{1}}<Q_{i-1}(q_{1}+1)+\dfrac{Q_{i}}{q_{1}}=\dfrac{Q_{i-1}}{q_{1}}(q_{1}(q_{1}+1)+q_{i}).
	\end{align*}
	It is sufficient to show  $q_{1}(q_{1}+1)+q_{i}<q_{1}q_{i}$. Indeed, note that $ q_{i}\ge q_{n_{0}+2}\ge q_{4}\ge q_{1}+5 $. Also, $q_{1}(q_{1}+1)<(q_{1}+5)(q_{1}-1)
	\le q_{i}(q_{1}-1)$ as required. 	

\noindent
(2)  If $ 2\nmid n_{0} $, then $ n_{0}+2 $ is odd and $ Q_{i}\ge Q_{n_{0}+2}>Q_{n_{0}+1}>d $ for $ i\ge n_{0}+2 $. Take $ N_{i} $ to be \eqref{eqn20} and the inequality follows by a similar argument to that given in the proof of part (1).
			
\noindent
(3) Suppose that $ q_{n_{0}+2}>Q_{n_{0}+1}$. Then $ q_{n_{0}+2}>d $. Since $ 2\nmid n_{0} $, $ 2\nmid \Omega_{D}^{-}(q_{i}Q_{n_{0}-1}) $ for every $i$. Also, $ D>0 $, so for each $ 1\le i\le n_{0} $, Lemma \ref{lem:newprime} implies that $q_{j_{i}}\mid d-q_{i}Q_{n_{0}-1}$ for some $j_i\geq n_0$ with $j_i\neq i$. Since $q_{j_i}<d<q_{n_0+2}$, we furthermore conclude that $ n_{0}\le j_{i}<n_{0}+2 $. In particular, taking $ i=n_{0}$, we conclude that
	\begin{align}\label{eq230}
	 q_{n_{0}+1}\mid d-Q_{n_{0}}.
	\end{align} 
	We claim that $ n_{0}=1 $. If not, then $ n_{0}\ge 3 $. Consider 
	\begin{equation*}
	\begin{cases}
	q_{j_{1}}\mid d-q_{1}Q_{n_{0}-1},   \\
    q_{j_{2}}\mid d-q_{2}Q_{n_{0}-1}.
	\end{cases}	
	\end{equation*}
	If $j_{1}=n_{0}+1 $, then $
	q_{n_{0}+1}\mid (d-q_{1}Q_{n_{0}-1})-(d-Q_{n_{0}})=Q_{n_{0}-1}(q_{n_{0}}-q_{1})$
	from \eqref{eq230}. This implies that $ q_{n_{0}+1}\mid q_{n_{0}}-q_{1}$, which is impossible because $q_{n_0+1}>q_{n_0}$. Similarly, if $ j_{2}=n_{0}+1 $, then we have $ q_{n_{0}+1}\mid q_{n_{0}}-q_{2} $, which is again impossible by the same argument. Hence $ j_{1}=j_{2}=n_{0} $ and it follows that 
\[
q_{n_{0}}\mid (d-q_{1}Q_{n_{0}-1})-(d-q_{2}Q_{n_{0}-1})=Q_{n_{0}-1}(q_{2}-q_{1}).
\]
Therefore $ q_{n_{0}}\mid q_{2}-q_{1} $, which is also impossible. We thus conclude that $n_0=1$.
	
Since $n_0=1$, Lemma \ref{lem22} (1) and Lemma \ref{lem:newprime} give $ q_{n_{0}+2}=q_{3}>d>q_{2} $, $q_{2}\mid d-q_{1}$ and $ q_{1}\mid d-q_{2}$. We may thus let $ \ell_{1},\ell_{2} $ be positive integers for which
	\begin{equation*}
	\begin{cases}
	\ell_{1}q_{1}=d-q_{2},  \\
	\ell_{2}q_{2}=d-q_{1}.
	\end{cases}
	\end{equation*}
Then $(\ell_{2}-1)q_{2}=(\ell_{1}-1)q_{1} $. It follows that $ q_{1}\mid \ell_{2}-1 $ from $ \gcd(q_{1},q_{2})=1 $. If $ \ell_{2}>1 $, then 
\[ 
\left(q_{1}+1\right)q_{2} \le \ell_{2}q_{2}=d-q_{1} 
\]
 and so, using $ d<Q_{n_{0}+1}=q_{1}q_{2}$ (as $ n_{0}=1 $),
\[
d+q_{2}+q_{1}<q_{1}q_{2}+q_{2}+q_{1}=(q_{1}+1)q_{2}+q_{1}\le d,
\]
 a contradiction. Hence $\ell_{2}=1$ and so 
\begin{equation}\label{eqn:dsum}
q_{1}+q_{2}=d
\end{equation}
 (and consequently $\ell_1=1$ as well). By Lemma \ref{lem:newprime}, there exists $q_i\in \mathbb{P}(D)\setminus\{q_1\}$ for which $q_i\mid d+q_1$. Note that $ q_{2}\nmid d+q_{1} $, since otherwise $ q_{2}\mid 2d $, which is a contradiction because $q_2>q_1\geq 2$ and $(D/q_2)=-1$ implies that $\gcd(q_2,d)=1$. Hence $i\geq 3$ and therefore
\[
q_{3}\le q_{i}\leq d+q_{1}<2d.
\]
From the assumption that $ q_{1}q_{2}=Q_2<q_{3}$ (since $n_0=1$), we deduce
\[
q_{1}q_{2}<q_{3}<2d=2(q_{1}+q_{2});
\]
that is $(q_{1}-2)(q_{2}-2)< 4 $. Hence we must have $ q_{1}=2 $ or $ q_{1}=3 $ and $ q_{2}=5 $. If $ q_{1}=3 $ and $ q_{2}=5 $, then $ D=d=8 $. One can compute $ q_{3}=11<q_{1}q_{2}=15$, a contradiction. Hence $ q_{1}=2 $ and we conclude from \eqref{eqn:dsum} that  $ q_{2}=d-2 $ and so 
\begin{equation}\label{eqn:q1q2}
q_{1}q_{2}=2d-4,
\end{equation}
 from which we see that $d>4$. Applying Lemma \ref{lem:newprime} with $s_{1}=q_{1}=2$ and $s_{2}=q_{2}=d-2$, we see that there exists $ q_{j} $ such that $ \gcd(q_{j},q_{1}q_{2})=1 $ ($ q_{1}=2 $) for which 
\[
q_{j}\mid 2d-(d-2)=d+2.
\]
Hence $ j\ge 3 $. Combining this with \eqref{eqn:q1q2} and the bound $Q_2<q_3$, it follows that 
\[
2d-4< q_{3}\le q_j\leq  d+2.
\]
So $d\le 6$, and hence $4<d\le 6$. Thus $ D=5 $, since $ D\equiv 0,1 \pmod{4}$ and $ D>0 $.

\noindent
(4) Suppose that $q_{n_{0}+2}>Q_{n_{0}+1} $. Then $  q_{n_{0}+2}>d $. We derive a contradiction by showing the following four assertions.

\noindent

\noindent
\begin{enumerate}[leftmargin=*,label={\rm(\alph*)}, align=left]	 
\item If $ (D/2)\neq -1 $, then $q_{n_{0}+2}>2d$.
\item We have $ d+Q_{n_{0}}\le Q_{n_{0}+1} $.
\item We have $q_{n_{0}+1}\mid d+Q_{n_{0}} $ and $ q_{n_{0}}\mid d-Q_{n_{0}-1} $.
\item If $(D/2)\neq -1$, then $q_{n_{0}}\nmid d+Q_{n_{0}-1}q_{n_{0}+1}$.
\end{enumerate}

Before proving the assertions (a)--(d), we demonstrate how (a)--(d) implies the claim. Assume that $D\neq -3$. If $ (D/2)\neq -1 $, as $ 2\nmid \Omega_{D}^{-}(Q_{n_{0}-1}q_{n_{0}+1}) $, Lemma \ref{lem:newprime} implies that there exists $j\geq n_0$ with $j\neq n_0+1$ for which 
\[
q_{j}\mid d+Q_{n_{0}-1}q_{n_{0}+1}.
\]
By (d), we see that $ j \ge n_{0}+2 $. Hence
\begin{equation}\label{eqn:qn0+2bnd}
q_{n_{0}+2}\le q_{j}\le d+Q_{n_{0}-1}q_{n_{0}+1}.
\end{equation}
Thus 
\begin{equation}\label{eqn:boundsqn0+2}
d+Q_{n_0}\overset{(b)}{\leq }Q_{n_0+1}<q_{n_0+2}\leq d+Q_{n_0-1}q_{n_0+1}.
\end{equation}

If $ d<Q_{n_{0}-1}q_{n_{0}+1} $, then 
\[
q_{n_{0}+2}\le d+Q_{n_{0}-1}q_{n_{0}+1}<2Q_{n_{0}-1}q_{n_{0}+1}<Q_{n_{0}+1}.
\]
This contradicts the original assumption, and hence we conclude that $Q_{n_{0}-1}q_{n_{0}+1}\leq d$, which together with \eqref{eqn:qn0+2bnd} implies that $q_{n_{0}+2}\leq 2d $. Thus by (a) we have $(D/2)=-1$.

So $q_1=2$ and $d$ is odd. By Lemma \ref{lem:newprime} we have $q_j\mid d+2Q_{n_0-1}$ for some $j$ with $j\geq n_0$. If $j\geq n_0+2$, then since $q_{n_0}\geq 2$ we have
\[
q_{n_0+2}\leq q_j\leq d+2Q_{n_0-1}\leq d+Q_{n_0},
\]
contradicting (b) (the first two inequalities in \eqref{eqn:boundsqn0+2} hold without the assumption $(D/2)\neq -1$). Thus $n_0\leq j\leq n_0+1$. If $j=n_0$, then since $q_{n_0}\mid d-Q_{n_0-1}$ by (c), we conclude that 
\[
q_{n_0}\mid (d+2Q_{n_0-1})-(d-Q_{n_0-1})=3Q_{n_0-1}.
\]
Since $\gcd(q_{n_0},Q_{n_0-1})=1$, we thus have $q_{n_0}=3$ and since $q_1=2$ we have $n_0=2$. This contradicts the assumption that $n_0$ is odd, however.

Therefore $q_{n_0+1}\mid d+2Q_{n_0-1}$. Since $q_{n_0+1}\mid d+Q_{n_0}$ by (c), we have 
\[
q_{n_0+1}\mid Q_{n_0-1}\left(q_{n_0}-2\right).
\]
Since $\gcd(q_{n_0+1},Q_{n_0-1})=1$, we see that $q_{n_0+1}\mid q_{n_0}-2$. If $q_{n_0}\neq 2$, then this contradicts $q_{n_0+1}>q_{n_0}$. Thus we see that $q_{n_0}=2$ and $n_0=1$. We hence have
\[
d<Q_{n_0+1}=2q_2
\]
and by (c) (and the fact that $Q_{n_0}=Q_1=2$) we also have 
\[
q_2\mid d+2.
\]
Writing $d+2=\ell q_2$ we have that $\ell$ is odd because $d+2$ is odd, and if $\ell\geq 3$ then 
\[
2q_2\leq (\ell-1) q_2=d+2-q_2\leq d-1<d
\]
contradicts the fact that $2q_2>d$. Thus $\ell=1$ and $q_2=d+2$. However, if $d>4$, then by Lemma \ref{lem:newprime} with $ s_{1}=1 $ and $s_2=4$, there exists $q_j\mid d-4$ with $j\geq 2$, so 
\[
d+2=q_2\leq q_j\leq d-4,
\] 
a contradiction. The only remaining case with $d\leq 4$ and odd is $d=3$ (i.e., $D=-3$).

We now move on to proving the assertions (a)--(d) under the assumption $ q_{n_{0}+2}>Q_{n_{0}+1} $ (and hence $ q_{n_{0}+2}>d $).

\noindent
(a) Suppose that $ q_{n_{0}+2}<2d $, so that 
\[
Q_{n_{0}+1}<q_{n_0+2}<2d.
\]
  As $ (D/2)\neq -1 $, $ 2\not\in \mathbb{P}(D) $. Since $ D<0 $ and $ n_{0}+1 $ is even, applying Lemma \ref{lem:newprime} with $ s_{1}=2 $ and $ s_{2}=Q_{n_{0}+1}$ (which is necessarily odd because $(D/2)\neq -1$), we deduce that $q_{n_{0}+2}\le 2d-Q_{n_{0}+1}$.  This implies that 
\[
2d=d+d<Q_{n_{0}+1}+q_{n_{0}+2}\le 2d,
\]
as $ d<Q_{n_{0}+1}<q_{n_{0}+2} $. This is a contradiction, and we conclude (a).
	 
\noindent
(b) If not, then $ Q_{n_{0}+1}<d+Q_{n_{0}} $ and 
	 \begin{equation}\label{eq25}
	  \begin{aligned}
	 (q_{n_{0}+1}-1)Q_{n_{0}}<d<Q_{n_{0}+1}<d+Q_{n_{0}}<2d,
	 \end{aligned}
	 \end{equation}	
	 as $ Q_{n_{0}}<d<Q_{n_{0}+1} $. Since $ D<0 $ and $ n_{0}+1 $ is even, if $(D/2)\neq -1$, then $Q_{n_0+1}$ is odd and $ q_{n_{0}+2}\le 2d-Q_{n_{0}+1}<2d$ by Lemma \ref{lem:newprime}. So $ (D/2)=-1 $ follows from the assertion (a). Thus $ q_{1}=2 $ and $ \gcd(d,2)=1 $.

	 Assume for contradiction that $ n_{0}\ge 3 $. Then for $i=1,2$ we have $ q_{n_{0}+1}-1>q_{i} $, and hence $ d-q_{i}Q_{n_{0}}>0 $ by the first inequality in \eqref{eq25}. By Lemma \ref{lem:newprime}, we have 
\[
q_{n_{0}+1}\mid d-q_{1}Q_{n_{0}}\qquad \text{ and }\qquad q_{n_{0}+1}\mid d-q_{2}Q_{n_{0}},
\]	 since $ q_{n_{0}+2}>d $, $ D<0 $ and $ \Omega_{D}^{-}(q_{i}Q_{n_{0}}) $ ($ i=1,2 $) is even. This implies that $ q_{n_{0}+1}\mid Q_{n_{0}}(q_{2}-q_{1}) $. Since $\gcd(q_{n_0+1},Q_{n_0})=1$, this implies that $q_{n_0+1}\mid q_2-q_1$, which is impossible because $q_{n_0+1}>q_2$. Hence $ n_{0}=1 $. Then \eqref{eq25} and $ q_{1}=2 $ gives
\[
2q_{2}-2=2(q_{2}-1)<d<2q_{2},
\]
and so $ d=2q_{2}-1  $, i.e., $ 2q_{2}=d+1 $. Note that $ q_{2}>2 $ and $ d+1 $ is even. Hence $ d\not=3,4 $ and thus $ d>4 $. Since $q_{3}=q_{n_0+2}>Q_{n_0+1}=2q_{2}>d$ by assumption, Lemma \ref{lem:newprime} implies that $ q_{2}\mid d-4 $ and so $ q_{2}\mid (d+1)-(d-4)=5 $. Accordingly, $ q_{2}=5 $ and we deduce $ 4<d<2q_{2}=10 $. But none of the sequences $ \{q_{i}\}_{D} $ for $ -10<D<-4 $ satisfies $ q_{1}=2 $ and $ q_{2}=5 $ at the same time. So the assertion (b) is true.

\noindent
(c) Observe that $  d+Q_{n_{0}}\le Q_{n_{0}+1}<q_{n_{0}+2}$ from the assertion (b) and the assumption $ Q_{n_{0}+1}<q_{n_{0}+2} $. Since $ 2\nmid n_{0} $, we have
	 \begin{equation}\label{eq21}
	 \begin{aligned}
	 q_{n_{0}+1}\mid d+Q_{n_{0}} 
	 \end{aligned}
	 \end{equation}
	 by Lemma \ref{lem:newprime}. Since $ n_{0}-1 $ is even, Lemma \ref{lem:newprime} implies that either $ q_{n_{0}+1}\mid d-Q_{n_{0}-1} $ or $ q_{n_{0}}\mid d-Q_{n_{0}-1} $,  as $ q_{n_{0}+2}>d $. Assume that $ q_{n_{0}+1}\mid d-Q_{n_{0}-1} $. From \eqref{eq21}, 
     \begin{align*}
     q_{n_{0}+1}\mid (d+Q_{n_{0}})-(d-Q_{n_{0}-1})=Q_{n_{0}-1}(q_{n_{0}}+1),
     \end{align*}
     which implies $ q_{n_{0}+1}\mid q_{n_{0}}+1 $. This may only occur if $q_{n_{0}}=2$ and $ q_{n_{0}+1}=3$, which together imply that $ n_{0}=1$ and $ q_{1}=2<d<q_{1}q_{2}=6 $. But there are no sequences $\{q_{i}\}_{D}$ associated with $ -6<D<-2 $ such that $ q_{1}=2 $ and $ q_{2}=3 $. So we must have 
     \begin{equation}\label{eq22}
     \begin{aligned}
     q_{n_{0}}\mid d-Q_{n_{0}-1}.
     \end{aligned}
     \end{equation}
     
\noindent
(d)  Suppose that $q_{n_{0}}\mid d+Q_{n_{0}-1}q_{n_{0}+1}$.  
    
 As $ (D/2)\not=-1 $, $ 2\not\in\mathbb{P}(D) $. From the assertion (a), $ q_{n_{0}+2}>2d $. Applying Lemma \ref{lem:newprime} with $ s_{1}=2 $ and $ s_{2}=Q_{n_{0}-1}$ (which is odd because $2\notin \mathbb{P}(D)$) gives $ q_{n_{0}}\mid 2d-Q_{n_{0}-1} $ or $ q_{n_{0}+1}\mid 2d-Q_{n_{0}-1} $. From \eqref{eq22}, it follows that $ q_{n_{0}+1}\mid 2d-Q_{n_{0}-1} $ (otherwise, $ q_{n_{0}}\mid d $).
    From \eqref{eq21}, we have
	\begin{align*}
	q_{n_{0}+1}\mid 2(d+Q_{n_{0}})-(2d-Q_{n_{0}-1})=Q_{n_{0}-1}(2q_{n_{0}}+1).
	\end{align*}
  	It follows that $q_{n_{0}+1}\mid 2q_{n_{0}}+1$. Since $2q_{n_0}+1$ is odd, there exists odd $\ell$ for which 
\[
\ell q_{n_0+1}= 2q_{n_{0}}+1<3q_{n_0}<3q_{n_0+1},
\]
and hence $\ell=1$ and 
	\begin{equation}\label{eq24}
	q_{n_{0}+1}=2q_{n_{0}}+1.
	\end{equation}
From the assumption $q_{n_{0}}\mid d+Q_{n_{0}-1}q_{n_{0}+1}$ and \eqref{eq22}, we also have
	\begin{align*}
	q_{n_{0}}\mid (d+Q_{n_{0}-1}q_{n_{0}+1})-(d-Q_{n_{0}-1})=Q_{n_{0}-1}(q_{n_{0}+1}+1).
	\end{align*}
Thus $ q_{n_{0}}\mid q_{n_{0}+1}+1 $, and combining this with \eqref{eq24}, we have 
	\begin{align*}
	q_{n_{0}}\mid q_{n_{0}+1}+1=2q_{n_{0}}+2.
	\end{align*} 
	Hence $ q_{n_{0}}=2 $, contradicting $ (D/2)\not=-1 $, and we conclude the assertion (d).
\end{proof}

We are now ready to prove Theorem \ref{thm1}.
\begin{proof}[Proof of Theorem \ref{thm1}]
 An easy computation together with Lemma \ref{lem23}(2) shows that Theorem \ref{thm1} is true for $ D=-3,-4,5$. For $ D\not=-3,-4,5 $, if $ q_{i_{0}+1} $ is the least prime greater than $d$, then $ i_{0}\ge i(D)+1 $ by Lemma \ref{lem22} and so $ q_{i+1}<q_{1}q_{2}\cdots q_{i} $ holds for $ i\ge i_{0} $ by Lemma \ref{lem23}. 
\end{proof}

\subsection{An application of Theorem \ref{thm1}}\label{sub23}
The following theorem, given by Dickson \cite[Theorem 97, p.\hskip 0.1cm 109]{dickson_modern_1939}, will be applied to the proof of Theorem \ref{thm5}.

\begin{thm}[Dickson]\label{thm231}
	Let $ a,b,c $ be positive integers with $ a\le b\le c $. Suppose that $ \gcd(a,b,c)=1 $ and no two of the $ a,b,c $ has an odd prime divisor in common. If there exists a positive odd integer $ n $ prime to $ abc $ such that the equation
	\begin{align*}
	ax^{2}+by^{2}+cz^{2}\equiv n\pmod{8}
	\end{align*} 
	is solvable and $ n $ is not represented by the form $ ax^{2}+by^{2}+cz^{2} $, then $ ax^{2}+by^{2}+cz^{2} $ is irregular.
\end{thm}

\begin{proof}[Proof of Theorem \ref{thm5}] Let $ f(x,y,z)=ax^{2}+by^{2}+cz^{2} $ be regular and assume for contradiction that $ c>4ab+3\delta_{ab,1} $. Choosing $ D=-4ab$, we have $ d=4ab $. Consider the prime sequence $ \{q_{i}\}_{-4ab}$ and let $i_0$ be chosen such that $q_{i_0+1}$ is the least in $ \{q_{i}\}_{-4ab}$ which is greater than $4ab+3\delta_{ab,1}$, i.e.,
	\begin{align*}
	q_{1}<q_{2}<\cdots<q_{i_{0}}<4ab+3\delta_{ab,1}<q_{i_{0}+1}<\cdots.
	\end{align*}
	For each $ i=1,2,\ldots $, we claim that $ q_{i} $ must be not represented by $ ax^{2}+by^{2} $ in $ \mathbb{Z} $. Otherwise, we have $ q_{i}=ax_{0}^{2}+by_{0}^{2} $ for some $ x_{0},y_{0}\in \mathbb{Z} $. Namely, $(2ax_{0})^{2}+4aby_{0}^{2}=4aq_{i}$. Take $ x_{1}=2ax_{0}$ and $ y_{1}=y_{0} $. Then $x_{1}^{2}+4aby_{1}^{2}\equiv 0 \pmod{q_{i}} $ is solvable. Since $q_i$ is an inert prime in the ring of integers of $\Q(\sqrt{-4ab})$ (as $(-4ab/q_{i})=-1$), this implies that $ x_{1}\equiv y_{1}\equiv 0\pmod{q_{i}} $. It follows that $q_{i}^{2}\mid x_{1}^{2}+4aby_{1}^{2}=4aq_{i}$, a contradiction.
	
	If $ \gcd(q_{j},c)=1 $ for some $ 1\le j\le i_{0} $, then $\gcd(q_{j},abc)=1$. For each odd prime $ p $, $ax^{2}+by^{2}+cz^{2}\equiv q_{j}\pmod{p^{t}}$ is solvable for any $ t\in\mathbb{N} $. Also, note that $ ax^{2}+by^{2}+cz^{2}\equiv q_{j}\pmod{2^{t}}$ is solvable for any $ t\in \mathbb{N} $ from the condition \eqref{eq12}. Hence
	$ q_{j} $ is locally represented by $ f $. However, $ q_{j}<4ab+3\delta_{ab,1}<c $ and $ q_{j} $ is not represented $ ax^{2}+by^{2} $ in $ \mathbb{Z} $, so $ q_{j} $ is not globally represented by $f $. It follows that $ f $ is irregular by Theorem \ref{thm231}, a contradiction.
	
	Suppose that $ q_{1}q_{2}\cdots q_{i_{0}}\mid c $. We assert that $ q_{1}q_{2}\cdots q_{i_{0}}q_{i_{0}+1}\mid c $. If not, then $ \gcd(q_{i_{0}+1},c)=1 $. It implies that $ q_{i_{0}+1} $ is locally represented by $ f $. But $q_{i_{0}+1}<q_{1}\cdots q_{i_{0}}\le c $ by Theorem \ref{thm1}. Applying the argument as above, we see that $ q_{i_{0}+1} $ is not globally represented by $ f $. Again, $ f $ is irregular by Theorem \ref{thm231}. Hence the claim is true.
	
	Repeating inductively the argument to $ q_{i} $ for $ i\ge i_{0}+2 $, we deduce that $ q_{1}q_{2}\cdots q_{i_{0}}q_{i_{0}+1}\cdots\mid c $. However, $ c $ is finite and so the assumption is false.
\end{proof}

%\newpage
\section*{Acknowledgments}
The author would like to thank his supervisor Dr. Ben Kane for his indispensable guidance in this paper, and also thank the referee for his/her helpful comments and suggestions.

%\section*{References}
%\bibliography{IM_ref.bib}
\bibliographystyle{plain}

\end{document}